\documentclass[11pt]{amsart}
\usepackage[margin=1in]{geometry}
\usepackage{amssymb}
\usepackage{amsthm}
\usepackage{amsmath}
\usepackage{mathrsfs}
\usepackage{amsbsy}
\usepackage{bm}
\usepackage{hyperref}
\usepackage{tikz}
\usepackage{array}
\usepackage{enumerate}
\usepackage{enumitem}
\usepackage{bbm}
\usepackage{comment}
\usepackage{mathtools}
\usepackage{makecell} 
\usepackage{colortbl}
\usepackage{xcolor}
\usepackage{graphicx}

\DeclareFontFamily{U}{mathx}{}
\DeclareFontShape{U}{mathx}{m}{n}{<-> mathx10}{}
\DeclareSymbolFont{mathx}{U}{mathx}{m}{n}
\DeclareMathAccent{\widecheck}{0}{mathx}{"71}

\definecolor{DangerousAxiomBlue}{RGB}{46,56,255}
\definecolor{DarkAxiomBlue}{RGB}{107,131,255}
\definecolor{AxiomBlue}{RGB}{153,187,255}
\definecolor{SafetyAxiomBlue}{RGB}{194,214,255}
\definecolor{DarkBlue}{RGB}{0,0,230}

\hypersetup{colorlinks=true, citecolor=AxiomBlue, linkcolor=DangerousAxiomBlue,urlcolor=DangerousAxiomBlue}

\usepackage{hhline}
\allowdisplaybreaks
\usepackage[noadjust]{cite}

\usepackage{caption}
\usepackage[noabbrev,capitalise,nameinlink]{cleveref}
\crefname{conjecture}{Conjecture}{Conjectures}
 
\newtheorem*{conjecture*}{Conjecture}

\newtheorem{theorem}{Theorem}[section]
\newtheorem{proposition}[theorem]{Proposition}

\newtheorem{lemma}[theorem]{Lemma}

\newtheorem{maintheorem}{Theorem}

\crefname{maintheorem}{Theorem}{Theorems}

\theoremstyle{definition}

\newtheorem{remark}[theorem]{Remark}

\usepackage{etoolbox}

\newcommand{\includeSymbol}[1]{\ensuremath{%
	\mathchoice
		{\raisebox{-.4mm}{\includegraphics[height=2.1ex]{#1}}}	
		{\raisebox{-.4mm}{\includegraphics[height=2.1ex]{#1}}}
		{\raisebox{-.3mm}{\includegraphics[height=1.6ex]{#1}}}
		{\raisebox{-.2mm}{\includegraphics[height=1ex]{#1}}}
}} 

\robustify{\includeSymbol}

\newcommand{\blam}{\boldsymbol{\lambda}}

\newcommand{\dfn}[1]{\textcolor{DarkAxiomBlue}{\emph{#1}}}

\begin{document}

\title[Proof of the Lyons--White Conjecture]{Proof of the Lyons--White Conjecture}
\subjclass[2010]{}

\author[]{Colin Defant}
\address[]{Axiom Math, 124 University Avenue, Palo Alto, CA 94301, USA}
\email{colin@axiommath.ai} 

\author[]{Ken Ono}
\address[]{Axiom Math, 124 University Avenue, Palo Alto, CA 94301, USA}
\email{ken@axiommath.ai}

\begin{abstract}
Let $D_n$ be the dihedral group of order $2n$. Consider a continuous-time random walk on $D_n$ driven by arbitrary symmetric rates whose support generates $D_n$. For $p\in[1,\infty]$, we say the pair $(D_n,p)$ is \emph{rate-monotonic} if for each fixed time $t$, the $\ell^p$-distance between the random walk's distribution at time $t$ and the uniform distribution is monotonically decreasing as a function of the rates. Lyons and White proved that $(D_n,2)$ and $(D_n,\infty)$ are rate-monotonic. Somewhat counterintuitively, they found several pairs $(D_n,p)$ with ${p\in[1,1.997]\cup[2.001,3.999]\cup[4.001,5.995]}$ that are not rate-monotonic, and they asked whether any such pairs exist with $p=4$ or $p=6$. We resolve their question, proving that $(D_n,2m)$ is rate-monotonic for all positive integers $m$ and $n$. In fact, we prove a generalization of this result to a broader family of groups that includes generalized dihedral groups, dicyclic groups, and generalized quaternion groups. In the other direction, we prove that for every real $p\geq 1$ that is not an even integer, there exists a positive integer $n$ such that $(D_n,p)$ is not rate-monotonic. 
The results of this paper were formally verified in Lean by AxiomProver assuming standard literature.     
\end{abstract}

\maketitle

\section{Introduction}
\subsection{Background}
Fix a finite group $G$. Let $\blam=(\lambda_s)_{s\in G}\in\mathbb R_{\geq 0}^G$ be a collection of nonnegative rates. We consider a natural continuous-time random walk ${\bf X}^{\blam}=(X_t^{\blam})_{t\geq 0}$ on $G$ driven by $\blam$. To each $s\in G$, we associate a Poisson process $T_s\subseteq\mathbb R_{\geq 0}$ with rate $\lambda_s$; the Poisson processes associated to different elements of $G$ are independent. We view $T_s$ as the set of times when a random clock associated to $s$ rings. The random walk starts at the identity element $e$ of $G$ at time $0$. Whenever the clock associated to $s$ rings, the walk transitions by multiplying its current state by $s$ on the right. In other words, if $t\in T_s$, then for all sufficiently small $\epsilon>0$, we have $X^{\blam}_{t+\epsilon}=X^{\blam}_ts$. Since $G$ is finite, the total rate $\sum_{s\in G}\lambda_s$ is finite, so there are almost surely only finitely many clock rings in every bounded time interval. Moreover, with probability $1$, no two distinct clocks ring at the same time, so these transitions determine the random walk uniquely almost surely. 

We will assume throughout the article that the collection $\blam$ is \dfn{symmetric}, meaning $\lambda_s=\lambda_{s^{-1}}$ for all $s\in G$. The \dfn{support} of $\blam$ is $\{s\in G:\lambda_s>0\}$. We will also assume that $\blam$ is \dfn{generating}, meaning the support of $\blam$ generates $G$; under this assumption, the random walk ${\bf X}^{\blam}$ is irreducible and has the uniform distribution as its unique stationary distribution. 

For $t\in\mathbb{R}_{\geq 0}$ and $g\in G$, let \[P^{\blam}_t(g)=\mathbb{P}(X_t^{\blam}=g),\] where $\mathbb{P}$ denotes probability. For $1\leq p<\infty$, let \[d^{\blam}_p(t)=\left(\sum_{g\in G}\left|P^{\blam}_t(g)-\frac{1}{|G|}\right|^p\right)^{1/p}.\] Let \[d^{\blam}_\infty(t)=\max_{g\in G}\left|P^{\blam}_t(g)-\frac{1}{|G|}\right|.\] Thus, $d^{\blam}_p(t)$ is the $\ell^p$-distance to stationarity of the random walk ${\bf X}^{\blam}$ at time $t$.  

One might expect that for fixed $p$ and $t$, the quantity $d_p^{\blam}(t)$ is monotonically decreasing as one increases the rates in $\blam$. This, however, is not always the case. Indeed, if we increase just the rates $\lambda_s$ and $\lambda_{s^{-1}}$ for some $s\in G$, then we will increase the frequency at which the random walk multiplies by $s$, but we will also increase the frequency at which it multiplies by $s^{-1}$. 

For $\blam=(\lambda_s)_{s\in G}$ and $\blam'=(\lambda_s')_{s\in G}$ in $\mathbb R_{\geq 0}^G$, let us write $\blam\leq\blam'$ if $\lambda_s\leq\lambda_s'$ for all $s\in G$. Lyons and White \cite{LW} were interested in understanding when the random walk does in fact display the aforementioned monotonicity in the rates. They gave special attention to Coxeter groups since they could leverage the combinatorial properties such groups possess. Namely, they used the Lifting Property of the Bruhat order \cite[Proposition~2.2.7]{BB} to prove the following result. 

\begin{theorem}[{\cite{LW}}]\label{thm:LW_Coxeter}
Let $G$ be a finite Coxeter group. Let $\blam,\blam'\in\mathbb R_{\geq 0}$ be symmetric rate collections such that $\blam\leq\blam'$, and assume the support of each of $\blam$ and $\blam'$ is the set of Coxeter generators of $G$. Then 
\[d_p^{\blam}(t)\geq d_p^{\blam'}(t)\] for all $1\leq p \leq \infty$ and all $t\geq 0$. 
\end{theorem} 

\cref{thm:LW_Coxeter} restricts the rate collections to those supported on the set of Coxeter generators of a Coxeter group. Lyons and White were also interested in the setting where the supports of the rate collections could be arbitrary generating sets. For $1\leq p\leq \infty$, let us say the pair $(G,p)$ is \dfn{rate-monotonic} if for all symmetric generating collections $\blam,\blam'\in\mathbb R_{\geq 0}^{G}$ satisfying $\blam\leq\blam'$, we have \[d_p^{\blam}(t)\geq d_p^{\blam'}(t)\] for all real $t\geq 0$.   

\begin{theorem}[{\cite{LW}}]\label{thm:LW_2}
For every finite group $G$, the pairs $(G,2)$ and $(G,\infty)$ are rate-monotonic. 
\end{theorem}

\subsection{Main Results}
In this article, we focus on the case where $G$ is the dihedral group $D_n$ of order $2n$. In this setting, Lyons and White observed a very surprising phenomenon. They found various positive integers $n$ and real numbers $p$ throughout the set $[1,1.997]\cup[2.001,3.999]\cup[4.001,5.995]$ for which they could prove $(D_n,p)$ is not rate-monotonic. \cref{thm:LW_2} explains why they could not find such examples with $p$ exactly equal to $2$. However, they also could not find any such examples for $p$ exactly equal to $4$ or $6$, even though they could find examples with $p$ very close to $4$ or $6$. They left open the problem of explaining these observations. We resolve this problem. 

\begin{maintheorem}\label{thm:main_forward}
For all positive integers $m$ and $n$, the pair $(D_n,2m)$ is rate-monotonic.     
\end{maintheorem}

\cref{thm:main_forward} extends to a more general family of groups. Let $A$ be a finite abelian group, and let $z\in A$ be an involution. We define the \dfn{inversion extension}
\begin{equation}\label{eq:inv_presentation}
G_{A,z}=\langle A,b\mid b^2=z,\ bab^{-1}=a^{-1}\text{ for every $a\in A$}\rangle.
\end{equation}
Every element of $G_{A,z}$ has a unique expression of the form $a$ or $ab$, where $a\in A$. Ordinary dihedral groups, generalized dihedral groups, dicyclic groups, and generalized quaternion groups are all inversion extensions of the form in \eqref{eq:inv_presentation}. 

\begin{maintheorem}\label{thm:main_general}
Let $A$ be a finite abelian group, and let $z\in A$ be an involution. For every positive integer $m$, the pair $(G_{A,z},2m)$ is rate-monotonic.  
\end{maintheorem}

In the converse direction, we obtain the following theorem, which further explicates the examples discovered by Lyons and White. 

\begin{maintheorem}\label{thm:main_converse}For every real number $p>1$ that is not an even integer, there exists a positive integer $n$ such that the pair $(D_n,p)$ is not rate-monotonic.     
\end{maintheorem} 

\begin{remark}
Lyons and White already note that $(D_{5},1)$ is not rate-montonic \cite{LW}, so one could extend the statement of \cref{thm:main_converse} to all $p\geq 1$. However, assuming $p>1$ allows us to give a simpler proof.   
\end{remark}

We suspect that there should be a generalization of \cref{thm:main_general} to an even broader family of groups. 

\begin{conjecture*}
If a finite group $G$ has an abelian subgroup of index $2$, then $(G,2m)$ is rate-monotonic for every positive integer $m$.    
\end{conjecture*}

\subsection{Relation to Prior Work} 
The problem of comparing the rates at which two Markov chains approach a common stationary distribution has a long history. Diaconis and Saloff-Coste developed general techniques for comparing eigenvalues and convergence rates of reversible Markov chains, with particular attention to random walks on finite groups \cite{DSCgroups,DSCreversible} (see also \cite{LP} for general background on mixing times). These methods generally yield quantitative bounds on convergence rather than an exact comparison between the distributions of two walks at every fixed time. Exact comparisons at fixed times are known in several more structured settings. Peres and Winkler proved that inserting extra updates cannot delay mixing for monotone spin systems started from an extremal configuration \cite{PW}, while Fill and Kahn developed comparison inequalities for stochastically monotone Markov chains that yield comparisons at fixed times \cite{FK}. 

There are also striking examples showing that mixing can be sensitive to seemingly mild changes in a random walk. Hermon and Kozma constructed pairs of Cayley graphs whose metrics differ by a uniformly bounded distortion but whose total-variation mixing times differ by an unbounded factor; they also constructed non-transitive networks in which small increases in certain edge weights can delay mixing \cite{HK}. These results concern mixing times along sequences of graphs, whereas rate-monotonicity asks for an exact comparison at each fixed time on a single finite group.

Random walks on dihedral groups have also been studied from the perspective of asymptotic mixing times. McCollum obtained lower and upper bounds for walks generated by one rotation and one reflection, as well as for walks generated by random subsets of a dihedral group \cite{McCollumLower,McCollumUpper}. More recently, Huang and Rao proved cutoff for random walks on finite dihedral groups driven by a growing number of uniformly chosen random generators and determined the cutoff time throughout a broad range of parameters \cite{HuangRao}. These works study the amount of time required for particular families of walks to approach stationarity as the size of the group tends to infinity, which differs from our set-up. 

Fourier-analytic and representation-theoretic methods for studying random walks on finite groups are classical (see, e.g., \cite{Dia,Terras}). The distinction between even and non-even exponents in \cref{thm:main_forward,thm:main_converse} has a particularly close analogue in the Hardy--Littlewood majorant problem. Hardy and Littlewood initiated the study of whether enlarging Fourier coefficients in absolute value can decrease an $\ell^p$-norm \cite{HL}. At positive even integers, the corresponding inequality with constant $1$ follows from Parseval's identity, while for non-even exponents it fails in general \cite{Boas,GR,MS,Krenedits}. This dichotomy is closely connected to our proofs. The proof of \cref{thm:main_forward} uses a finite cyclic version of the positive majorant argument for even exponents, while the construction in the proof of \cref{thm:main_converse} exploits a negative Fourier coefficient of the type that appears in counterexamples to the Hardy--Littlewood majorant property.

\subsection{Outline}
 \cref{sec:forward} is devoted to proving \cref{thm:main_general}, which implies \cref{thm:main_forward} as a special case. \cref{sec:converse} is devoted to proving \cref{thm:main_converse}. In the appendix, we briefly discuss the formal verification of the theorems in this paper.

\section{Rate-Monotonicity at Even Integers}\label{sec:forward} 

In this section, we fix a finite abelian group $A$ and an involution $z\in A$, and we let $G=G_{A,z}$ be the inversion extension presented in \eqref{eq:inv_presentation}. Every element of $G_{A,z}$ has a unique expression of the form $a$ or $ab$, where $a\in A$. For every $a\in A$, we have
\[(ab)^2=z\quad\text{and}\quad (ab)^{-1}=zab.\]
When $A=\langle r\rangle$ is cyclic of order $n$ and $z=e$, we recover the presentation
\begin{equation}\label{eq:presentation}
D_n=\langle r,b\mid r^n=b^2=e,\ brb=r^{-1}\rangle.
\end{equation}

We begin by interpreting the random walk ${\bf X}^{\blam}$ via the group algebra of $G$. For \[x=\sum_{g\in G}x_g g\in\mathbb C[G],\] let $x^*=\sum_{g\in G}\overline{x_g}g^{-1}$. We identify $x$ with the operator on $\ell^2(G)$ given by left multiplication by $x$. We say $x$ is \dfn{positive} if this operator is positive semidefinite. If $x$ is positive and $\theta>0$, then the Spectral theorem allows us to define a fractional power $x^\theta\in\mathbb C[G]$. When $x$ has real coefficients in the standard basis of the group algebra, so does $x^\theta$.

The element 
\[\pi=\frac{1}{|G|}\sum_{g\in G}g\]
is a central idempotent, and for $x\in\mathbb C[G]$, the element $\pi x$ is the orthogonal projection of $x$ onto the space of constant functions. Given a symmetric rate collection $\blam$, define
\[\Delta_{\blam}=\sum_{s\in G}\lambda_s(e-s).\] Let us also define the \dfn{heat element} $h_t^{\blam}=e^{-t\Delta_{\blam}}$ and the \dfn{centered heat element} $a_t^{\blam}=h_t^{\blam}-\pi$ of the random walk ${\bf X}^{\blam}$ at time $t$. 

\begin{lemma}\label{lem:heat}
For every symmetric rate collection $\blam$ and every $t\geq 0$, we have
\[h_t^{\blam}=\sum_{g\in G}P_t^{\blam}(g)g.\]
Moreover, $a_t^{\blam}$ is positive, and
\[\left(a_t^{\blam}\right)^\theta=e^{-\theta t\Delta_{\blam}}(e-\pi)\]
for every $\theta>0$.
\end{lemma}

\begin{proof}
Because $\blam$ is symmetric, we have
\[
2\Delta_{\blam}
=\sum_{s\in G}\lambda_s(e-s)^*(e-s),
\]
so $\Delta_{\blam}$ is positive and self-adjoint. Let $\Lambda=\sum_{s\in G}\lambda_s$ and $w=\sum_{s\in G}\lambda_s s$. Since $\Delta_{\blam}=\Lambda e-w$, we have
\[
h_t^{\blam}
=e^{-t\Lambda}\sum_{q\geq 0}\frac{t^q}{q!}w^q.
\]
From this, we see that the expansion of $h_t^{\blam}$ in the standard basis of $\mathbb C[G]$ agrees with the description of the random walk ${\bf X}^{\blam}$ via Poisson processes. This proves the first statement. 

Because the coefficients of $h_t^{\blam}$ sum to $1$, we have $h_t^{\blam}\pi=\pi$. Hence, $a_t^{\blam}$ vanishes on the image of $\pi$ and agrees with the positive definite operator $h_t^{\blam}$ on the kernel of $\pi$. This proves that $a_t^{\blam}$ is positive. The same orthogonal decomposition shows that $\left(a_t^{\blam}\right)^\theta$ vanishes on the image of $\pi$ and agrees with $e^{-\theta t\Delta_{\blam}}$ on the kernel of $\pi$, which yields the final statement.
\end{proof}

Let $\widehat A$ be the character group of $A$. For a sequence $f=(f_a)_{a\in A}$ of complex numbers, we consider the unnormalized Fourier transform $\widehat f$ defined by 
\[\widehat f(\chi)=\sum_{a\in A}f_a\chi(a)\]
for $\chi\in\widehat A$. 

\begin{lemma}\label{lem:fourier_majorant}
Let $f=(f_a)_{a\in A}$ and $u=(u_a)_{a\in A}$ be sequences of complex numbers. Suppose $\widehat u(\chi)$ is a nonnegative real number and $|\widehat f(\chi)|\leq\widehat u(\chi)$ for every $\chi\in\widehat A$. For every positive integer $m$, we have
\[
\sum_{a\in A}|f_a|^{2m}\leq\sum_{a\in A}|u_a|^{2m}.
\]
\end{lemma}

\begin{proof}
The product formula for the Fourier transform gives
\[
\widehat{f^m}(\chi)
=\frac{1}{|A|^{m-1}}
\sum_{\chi_1\cdots\chi_m=\chi}
\widehat f(\chi_1)\cdots\widehat f(\chi_m),
\]
where $f^m$ denotes the pointwise power. It follows from the hypotheses that
\[
\left|\widehat{f^m}(\chi)\right|
\leq
\frac{1}{|A|^{m-1}}
\sum_{\chi_1\cdots\chi_m=\chi}
\widehat u(\chi_1)\cdots\widehat u(\chi_m)
=\widehat{u^m}(\chi).
\]
In particular, $\widehat{u^m}(\chi)$ is nonnegative. Parseval's identity now implies that 
\[\sum_{a\in A}|f_a|^{2m}=\frac{1}{|A|}\sum_{\chi\in\widehat A}|\widehat{f^m}(\chi)|^2
\leq\frac{1}{|A|}\sum_{\chi\in\widehat A}|\widehat{u^m}(\chi)|^2
=\sum_{a\in A}|u_a|^{2m}. \qedhere \] 
\end{proof}

The following proposition is the main ingredient in the proof of \cref{thm:main_general}.

\begin{proposition}\label{prop:reflection_inequality}
Let $x=\sum_{g\in G}x_g g\in\mathbb R[G]$ be positive, let $c\in G\setminus A$, and let $0<\theta<1$. Write
\[x^\theta c x^{1-\theta}
=\sum_{g\in G}w_g g.\]
Then for every positive integer $m$, we have
\[\sum_{g\in G}x_g^{2m-1}w_g\leq\sum_{g\in G}x_g^{2m}.\]
\end{proposition}

\begin{proof}
Every element $c\in G\setminus A$ satisfies $cac^{-1}=a^{-1}$ for every $a\in A$ and $c^2=z$, so we may replace $b$ with $c$ in the presentation of $G$. Thus, we can assume $c=b$. Write
\[x=\sum_{a\in A} u_a a+\sum_{a\in A}v_a ab
\quad\text{and}\quad
x^\theta b x^{1-\theta}
=\sum_{a\in A}q_a a+\sum_{a\in A}y_a ab;\]
all of these coefficients are real. Since $x$ is positive, it is self-adjoint. Using $(ab)^{-1}=zab$, we obtain
\begin{equation}\label{eq:self_adjoint_coefficients}
u_{a^{-1}}=u_a\quad\text{and}\quad v_{za}=v_a
\end{equation}
for every $a\in A$.

For each $\chi\in\widehat A$, note that $\chi(z)\in\{1,-1\}$, and let
\[
U_\chi=\sum_{a\in A} u_a\chi(a),\qquad
V_\chi=\sum_{a\in A}v_a\chi(a),\qquad
Q_\chi=\sum_{a\in A}q_a\chi(a),\qquad
Y_\chi=\sum_{a\in A}y_a\chi(a).
\]
The first identity in \eqref{eq:self_adjoint_coefficients} shows that $U_\chi$ is real, while the second implies that
 $V_\chi=\chi(z) V_\chi$.
In particular, $V_\chi=0$ whenever $\chi(z)=-1$.

Consider the unitary representation $\rho_\chi$ of $G$ defined by 
\[\rho_\chi(a)=\begin{pmatrix}\chi(a) &0 \\
0 &\overline{\chi(a)}
\end{pmatrix}
\quad\text{and}\quad
\rho_\chi(b)=B_\chi:=
\begin{pmatrix}
0 &\chi(z) \\
1 &0
\end{pmatrix}.\]
This is indeed a representation since $B_\chi^2=\chi(z) I=\rho_\chi(z)$ and $B_\chi\rho_\chi(a)B_\chi^{-1}=\rho_\chi(a^{-1})$. We have
\[\rho_\chi(x)=M_\chi:=\begin{pmatrix}
U_\chi &V_\chi \\
\overline{V_\chi} &U_\chi \end{pmatrix},\]
where we used the fact that $V_\chi=\chi(z) V_\chi$. The representation $\rho_\chi$ is a direct sum of irreducible representations, each of which occurs in the regular representation. Hence, the positivity of $x$ implies that $M_\chi$ is positive semidefinite. Consequently, $U_\chi\geq |V_\chi|\geq 0$. 

We wish to show that 
\begin{equation}\label{eq:f}
\mathrm{Re}(Q_\chi)=\mathrm{Re}(V_\chi)\quad\text{and}\quad |Y_\chi|\leq U_\chi
\end{equation}
for every $\chi\in\widehat A$. 

First, suppose $\chi(z)=1$. Write $V_\chi=\zeta |V_\chi|$, where $|\zeta|=1$; if $V_\chi=0$, choose $\zeta$ arbitrarily. The eigenvalues of $M_\chi$ are $\mu_{+}=U_\chi+|V_\chi|$ and $\mu_{-}=U_\chi-|V_\chi|$. For $\eta\in\{\theta,1-\theta\}$, let
\[c_\eta=\frac{\mu_{+}^\eta+\mu_{-}^\eta}{2}
\quad\text{and}\quad
d_\eta=\frac{\mu_{+}^\eta-\mu_{-}^\eta}{2}.\]
These numbers are nonnegative, and
\[M_\chi^\eta=
\begin{pmatrix}
c_\eta&\zeta d_\eta\\
\overline\zeta d_\eta&c_\eta
\end{pmatrix}.\]
Since $M_\chi^\theta M_\chi^{1-\theta}=M_\chi$, we have
\[c_\theta c_{1-\theta}+d_\theta d_{1-\theta}=U_\chi
\quad\text{and}\quad c_\theta d_{1-\theta}+d_\theta c_{1-\theta}=|V_\chi|. \]
Because $\rho_\chi$ is a unitary representation, we have
\[\rho_\chi\left(x^\theta b x^{1-\theta}\right)
=M_\chi^\theta
\begin{pmatrix}
0&1\\
1&0
\end{pmatrix}
M_\chi^{1-\theta}.\]
On the other hand,
\[\rho_\chi\left(x^\theta b x^{1-\theta}\right)=
\begin{pmatrix}
Q_\chi&Y_\chi\\
\overline Y_\chi&\overline Q_\chi
\end{pmatrix}.\]
It follows that
\[Q_\chi=\zeta d_\theta c_{1-\theta}
+\overline\zeta c_\theta d_{1-\theta}
\quad\text{and}\quad
Y_\chi=c_\theta c_{1-\theta}+\zeta^2d_\theta d_{1-\theta}.\]
This proves \eqref{eq:f} in this case. 

Now suppose $\chi(z)=-1$. Then $V_\chi=0$, so $M_\chi=U_\chi I$. We have
\[\rho_\chi\left(x^\theta b x^{1-\theta}\right)
=U_\chi
\begin{pmatrix}
0&-1\\
1&0
\end{pmatrix},\]
so
\[\rho_\chi\left(x^\theta b x^{1-\theta}\right)=
\begin{pmatrix}
Q_\chi&-Y_\chi\\
\overline Y_\chi&\overline Q_\chi
\end{pmatrix}.\]
This implies that $Q_\chi=0$ and $Y_\chi=U_\chi$, which yields \eqref{eq:f} in this case as well. 

If $f=(f_a)_{a\in A}$ is a sequence of real numbers, then the Fourier transform of the sequence ${((f_a+f_{a^{-1}})/2)_{a\in A}}$ is $(\mathrm{Re}(\widehat f(\chi)))_{\chi\in\widehat A}$. Fourier inversion and the first identity in \eqref{eq:f} yield that 
\[\frac{q_a+q_{a^{-1}}}{2}=\frac{v_a+v_{a^{-1}}}{2}\]
for every $a\in A$. Since $u_{a^{-1}}=u_a$, it follows that
\[\sum_{a\in A}u_a^{2m-1}q_a=\sum_{a\in A}u_a^{2m-1}v_a.\]
Moreover, \cref{lem:fourier_majorant} applies to the sequences $y$ and $u$ because $|Y_\chi|\leq U_\chi$ and $U_\chi\geq 0$. Hence, 
\[\left(\sum_{a\in A}|y_a|^{2m}\right)^{1/(2m)}\leq\left(\sum_{a\in A}|u_a|^{2m}\right)^{1/(2m)}.\]

Let $R=(\sum_{a\in A}|u_a|^{2m})^{1/(2m)}$ and $S=(\sum_{a\in A}|v_a|^{2m})^{1/(2m)}$. H\"older's inequality implies that 
\begin{align*}
\sum_{g\in G}x_g^{2m-1}w_g
&=\sum_{a\in A} u_a^{2m-1}q_a+\sum_{a\in A}v_a^{2m-1}y_a \\
&=\sum_{a\in A} u_a^{2m-1}v_a+\sum_{a\in A}v_a^{2m-1}y_a \\
&\leq R^{2m-1}S+S^{2m-1}R \\
&\leq R^{2m}+S^{2m}
=\sum_{g\in G}x_g^{2m}.
\end{align*}
Indeed, the penultimate inequality follows from the fact that $(R-S)(R^{2m-1}-S^{2m-1})\geq 0$. This proves the proposition.
\end{proof}

We can now prove \cref{thm:main_general}. We will use the finite-dimensional Duhamel identity
\[\frac{\mathrm d}{\mathrm d\alpha}e^{-t(C+\alpha D)}=-\int_0^te^{-\tau(C+\alpha D)}De^{-(t-\tau)(C+\alpha D)}\,\mathrm d\tau.\]
This follows by differentiating $e^{-\tau(C+(\alpha+\epsilon)D)}e^{-(t-\tau)(C+\alpha D)}$ with respect to $\tau$, integrating from $0$ to $t$, dividing by $\epsilon$, and letting $\epsilon\to 0$.

\begin{proof}[Proof of \cref{thm:main_general}]
Fix a positive integer $m$, and let $\blam\leq\blam'$ be symmetric generating rate collections on $G$. The rate at the identity has no effect on the random walk. The claim is immediate when $t=0$, so fix $t>0$. The remaining elements of $G$ split into inverse-orbits, so it suffices to prove that the quantity $(d_{2m}^{\blam}(t))^{2m}$ cannot increase when the common rate associated to one inverse-orbit $\mathcal O$ is increased. 

First, suppose $\mathcal O\subseteq A$. Fix $c\geq 0$, and let $\widetilde{\blam}$ be obtained from $\blam$ by increasing the rate assigned to every element of $\mathcal O$ by $c$. The element
\[\xi_{\mathcal O}=\sum_{s\in\mathcal O}(e-s)\]
is central in $\mathbb R[G]$. Indeed, it commutes with every element of $A$, while conjugation by an element of $G\setminus A$ sends every element of $A$ to its inverse and, as a consequence, permutes the elements of $\mathcal O$. Hence,
\[h_t^{\widetilde{\blam}}=e^{-tc\xi_{\mathcal O}}h_t^{\blam}.\]
Let $\kappa=e^{-tc\xi_{\mathcal O}}$. This is the heat element for a random walk driven by a rate collection supported on $\mathcal O$, so its coefficients $(\kappa_x)_{x\in G}$ are nonnegative and sum to $1$. Since $\kappa\pi=\pi$, we have $a_t^{\widetilde{\blam}}=\kappa a_t^{\blam}$. Writing $a_t^{\blam}=\sum_{g\in G} a_g g$ and $a_t^{\widetilde{\blam}}=\sum_{g\in G} \widetilde a_g g$, we can use Jensen's inequality to find that
\[\sum_{g\in G}\left|\widetilde a_g\right|^{2m}=\sum_{g\in G}\left|\sum_x\kappa_xa_{x^{-1}g}\right|^{2m}\leq\sum_{g\in G}\sum_x\kappa_x|a_{x^{-1}g}|^{2m}=\sum_{g\in G}|a_g|^{2m}.\] 
This proves the theorem in this case. 

We now prove the theorem in the case where $\mathcal O\subseteq G\setminus A$. Choose $c\in\mathcal O$. Then $\mathcal O=\{c,c^{-1}\}$, with repeated elements removed. Hold all rates assigned to elements of $G\setminus\mathcal O$ fixed, and let $\alpha$ denote the common rate assigned to the elements of $\mathcal O$. We think of $\alpha$ as a variable that we will increase; our goal is to prove that the centered $\ell^{2m}$-norm at a fixed time $t>0$ does not increase as we increase $\alpha$. Write
\[\Delta_\alpha=\Delta_0+\alpha\sum_{s\in\mathcal O}(e-s),\qquad h_\alpha=e^{-t\Delta_\alpha},\qquad a_\alpha=h_\alpha-\pi=h_\alpha(e-\pi). \]
Duhamel's identity tells us that 
\[ \frac{\mathrm d}{\mathrm d\alpha}h_\alpha = -\int_0^t e^{-\tau\Delta_\alpha}\left(\sum_{s\in\mathcal O}(e-s)\right)e^{-(t-\tau)\Delta_\alpha}\,\mathrm d\tau.\]
Multiplying on the right by $e-\pi$ yields the identity
\[\frac{\mathrm d}{\mathrm d\alpha}a_\alpha=-|\mathcal O|ta_\alpha+\sum_{s\in\mathcal O}\int_0^te^{-\tau\Delta_\alpha}se^{-(t-\tau)\Delta_\alpha}(e-\pi)\,\mathrm d\tau.\]
Since $s\pi=\pi$ and $\pi$ is central, each integrand in the second term will not change if we insert $e-\pi$ immediately before $s$. Therefore, invoking \cref{lem:heat} and using the substitution $\theta=\tau/t$, we find that
\[\frac{\mathrm d}{\mathrm d\alpha}a_\alpha=t\int_0^1\left(\sum_{s\in\mathcal O}a_\alpha^\theta s a_\alpha^{1-\theta}-|\mathcal O|a_\alpha\right)
\,\mathrm d\theta.\]
The values at $\theta=0$ and $\theta=1$ do not affect the integral.

Write $a_\alpha=\sum_{g\in G} a_{\alpha,g}g$, and let $\Phi(\alpha)=\sum_{g\in G} a_{\alpha,g}^{2m}$. Differentiating and using \cref{prop:reflection_inequality} for each $s\in\mathcal O$ shows that 
\[
\Phi'(\alpha) = 2mt\int_0^1
\left(\sum_{s\in\mathcal O}\sum_{g\in G}a_{\alpha,g}^{2m-1}\left(a_\alpha^\theta s a_\alpha^{1-\theta}\right)_g-|\mathcal O|\sum_{g\in G} a_{\alpha,g}^{2m}
\right)
\,\mathrm d\theta \leq 0.\]
Hence, increasing the common rate associated to an inverse-orbit in $G\setminus A$ cannot increase the centered $\ell^{2m}$-norm.

We can pass from $\blam$ to $\blam'$ by increasing the rates on one inverse-orbit at a time. Applying the preceding two arguments at each step shows that 
\[\sum_{g\in G}
\left|P_t^{\blam'}(g)-\frac{1}{|G|}\right|^{2m}\leq\sum_{g\in G}\left|P_t^{\blam}(g)-\frac{1}{|G|}\right|^{2m}.\]
Taking $(2m)$-th roots yields the desired inequality $d_{2m}^{\blam'}(t)\leq d_{2m}^{\blam}(t)$.   
\end{proof}

\begin{proof}[Proof of \cref{thm:main_forward}]
Take $A=\langle r\rangle\cong \mathbb Z/n\mathbb Z$ and $z=e$ in \cref{thm:main_general}. Then $G_{A,z}=D_n$, so \cref{thm:main_forward} follows. 
\end{proof}

\section{Lack of Rate-Monotonicity Away from Even Integers}\label{sec:converse} 

In this section, we fix a real number $p\geq 1$ that is not an even integer. For $y\in\mathbb R$, let
\[
J_p(y)=\begin{cases}
|y|^{p-2}y & \text{if $y\neq 0$}; \\
0 & \text{if $y=0$}.
\end{cases}
\]
Thus, $J_1(y)$ is the sign of $y$ when $y\neq 0$. 

\begin{lemma}\label{lem:Fourier}
There exist odd positive integers $K$ and $n$ satisfying $n>2K$ such that
\[\sum_{j=0}^{n-1}J_p\left(\cos\left(\frac{2\pi j}{n}\right)\right)
\cos\left(\frac{2\pi Kj}{n}\right)<0.\]
\end{lemma}

\begin{proof}
Choose a nonnegative integer $m$ such that $2m<p<2m+2$, and let $K=2m+3$. Define
\[\gamma_K(p)=\frac{1}{2\pi}\int_0^{2\pi}J_p(\cos t)\cos(Kt)\,\mathrm dt. \]
Because $K$ is odd, we have $\gamma_K(p)=\frac{2}{\pi}I_K(p)$, where
\[I_K(p)=\int_0^{\pi/2}\cos^{p-1}(t)\cos(Kt)\,\mathrm dt. \]
For every odd positive integer $L$, we have
\begin{equation}\label{eq:I_recurrence}
I_{L+2}(p)=\frac{p-L-1}{p+L+1}I_L(p).
\end{equation}
Indeed, we can combine the identity
\[\cos((L+2)t)-\cos(Lt)=-2\sin((L+1)t)\sin t\]
together with integration by parts to obtain the identity
\[I_{L+2}(p)-I_L(p)
=-\frac{2(L+1)}{p}\int_0^{\pi/2}\cos^p(t)\cos((L+1)t)\,\mathrm dt.\]
On the other hand, the product-to-sum formula implies that 
\[2\int_0^{\pi/2}\cos^p(t)\cos((L+1)t)\,\mathrm dt
=I_{L+2}(p)+I_L(p). \]
Combining the preceding two identities yields \eqref{eq:I_recurrence}. Since \[I_1(p)=\int_0^{\pi/2}\cos^p(t)\,\mathrm dt>0,\]
we find that \[I_{2m+3}(p)=I_1(p)\prod_{\ell=1}^{m+1}\frac{p-2\ell}{p+2\ell}<0. \]
Indeed, the factors indexed by $1\leq\ell\leq m$ are positive, while $p-2(m+1)<0$. Hence, $\gamma_K(p)<0$. 

As $n$ tends to infinity through odd positive integers, the normalized Riemann sums
\[\frac{1}{n}\sum_{j=0}^{n-1}
J_p\left(\cos\left(\frac{2\pi j}{n}\right)\right)
\cos\left(\frac{2\pi Kj}{n}\right)\]
converge to $\gamma_K(p)$. This remains true when $p=1$ because the integrand is Riemann integrable. Therefore, we can choose an odd integer $n>2K$ for which the sum in the statement of the lemma is negative. 
\end{proof}

\begin{lemma}\label{lem:heat_realization}
Let $n$ be a positive integer, and let $x\in\mathbb R[D_n]$ be positive. Suppose $x\pi=0$ and $x$ is positive definite on the kernel of $\pi$. Then there exist a symmetric generating rate collection $\blam\in\mathbb R_{\geq 0}^{D_n}$ and a real number $R>0$ such that \[a_1^{\blam}=e^{-R}x. \]
\end{lemma}

\begin{proof}
Define a self-adjoint element $H\in\mathbb R[D_n]$ by letting $H$ vanish on the image of $\pi$ and letting $H$ agree with $-\log x$ on the kernel of $\pi$. Then $e^{-H}=\pi+x$. Write $H=\sum_{g\in D_n}h_g g$. Since $H$ is self-adjoint, we have $h_{g^{-1}}=h_g$ for every $g\in D_n$. Moreover, $H\pi=0$, so $\sum_{g\in D_n}h_g=0$. Choose $R>0$ large enough so that \[\lambda_g=\frac{R}{2n}-h_g>0\] for every $g\in D_n\setminus\{e\}$, and let $\lambda_e=0$. Then $\blam=(\lambda_g)_{g\in D_n}$ is symmetric and generating, and
\[\Delta_{\blam} = \sum_{g\in D_n\setminus\{e\}}\lambda_g(e-g)=H+R(e-\pi).\] The element $\Delta_{\blam}$ vanishes on the image of $\pi$ and agrees with $H+R$ on the kernel of $\pi$. Therefore,
\[
h_1^{\blam}=e^{-\Delta_{\blam}}=\pi+e^{-R}x.
\]
This proves the desired identity.
\end{proof}

We are now in a position to prove \cref{thm:main_converse}.

\begin{proof}[Proof of \cref{thm:main_converse}]
Suppose the element \[x=\sum_{g\in D_n}x_g g\in\mathbb R[D_n]\] is positive and satisfies $x\pi=0$. Define
\[\mathcal E_p(x;b)=\int_0^1 \sum_{g\in D_n} J_p(x_g)\left(\left(x^\theta b x^{1-\theta}\right)_g-x_g\right)\mathrm d\theta.\]
The values of the integrand at $\theta=0$ and $\theta=1$ will not matter. Suppose $x$ is the centered heat element at time $1$ for a walk on $D_n$. For $\alpha\geq 0$, let $x_\alpha$ be the centered heat element obtained by increasing the rate of $b$ by $\alpha$ while holding all other rates fixed. The calculation via the Duhamel identity in the proof of \cref{thm:main_general} yields the identity
\[\left.\frac{\mathrm d}{\mathrm d\alpha}\right|_{\alpha=0}x_\alpha=\int_0^1\left(x^\theta b x^{1-\theta}-x\right)\mathrm d\theta.\]
Consequently, since $p>1$, we have
\begin{equation}\label{eq:infinitesimal_converse}
\left.\frac{\mathrm d}{\mathrm d\alpha}\right|_{\alpha=0}\sum_{g\in D_n}|x_{\alpha,g}|^p 
=p\,\mathcal E_p(x;b).
\end{equation}
Thus, it suffices to construct a centered heat element $x$ such that $\mathcal E_p(x;b)>0$.

Let $K$ and $n$ be as in \cref{lem:Fourier}, and use the presentation of $D_n$ in \eqref{eq:presentation}. Let $\omega=e^{2\pi i/n}$. Write
\[x=\sum_{j=0}^{n-1}u_jr^j+\sum_{j=0}^{n-1}v_jr^jb,\]
and let
\[U_k=\sum_{j=0}^{n-1}u_j\omega^{kj}
\quad\text{and}\quad V_k=\sum_{j=0}^{n-1}v_j\omega^{kj}.\] As in the proof of \cref{prop:reflection_inequality}, the $k$-th Fourier block of $x$ is
\[M_k=
\begin{pmatrix}
U_k&V_k\\
\overline{V_k}&U_k
\end{pmatrix}.
\]
We interpret the subscripts on $U_k$, $V_k$, and $M_k$ modulo $n$.

Since $n>2K$, the residues $\pm1$ and $\pm K$ are distinct modulo $n$. For $\epsilon>0$, let $x_\epsilon$ be the element whose Fourier data are given by
\[U_1=U_{-1}=1,\qquad V_1=V_{-1}=1,\]
\[U_K=U_{-K}=\epsilon,\qquad
V_K=i\epsilon,\qquad
V_{-K}=-i\epsilon,\]
\[U_k=V_k=0 \quad\text{for all }k\not\in\{\pm 1,\pm K\}.\] Its nonzero Fourier blocks are
\[M_{\pm 1}=
\begin{pmatrix}
1&1 \\ 
1&1
\end{pmatrix},\qquad
M_K=\epsilon
\begin{pmatrix}
1&i \\
-i&1
\end{pmatrix},\qquad
M_{-K}=\epsilon
\begin{pmatrix}
1&-i \\ 
i&1
\end{pmatrix}.\]
Each of these matrices is positive semidefinite, so $x_\epsilon$ is positive. Using Fourier inversion, we find that 
\[u_j=\frac{2}{n}\left(\cos\left(\frac{2\pi j}{n}\right)+\epsilon\cos\left(\frac{2\pi Kj}{n}\right)\right)\quad\text{and}\quad v_j=\frac{2}{n}\left(\cos\left(\frac{2\pi j}{n}\right)
+\epsilon\sin\left(\frac{2\pi Kj}{n}\right)\right). \]  
Moreover, $U_0=V_0=0$, so $x_\epsilon\pi=0$.

In the $k$-th Fourier block, the reflection $b$ is represented by the matrix   
\[B= \begin{pmatrix}
0&1\\
1&0
\end{pmatrix}. \] 
Let \[P_+=\frac{1}{2}
\begin{pmatrix}
1&1\\
1&1
\end{pmatrix},\qquad
P_i=\frac{1}{2}
\begin{pmatrix}
1&i\\
-i&1
\end{pmatrix},\qquad
P_{-i}=\overline{P_i}.
\]
We have $P_+BP_+=P_+$ and $P_iBP_i=P_{-i}BP_{-i}=0$. Since $M_{\pm 1}=2P_+$, $M_K=2\epsilon P_i$, and $M_{-K}=2\epsilon P_{-i}$, it follows that 
\[M_{\pm 1}^\theta BM_{\pm 1}^{1-\theta}=M_{\pm 1}
\quad\text{and}\quad M_{\pm K}^\theta BM_{\pm K}^{1-\theta}=0\] for all $0<\theta<1$. Therefore, if
\[x_\epsilon^\theta b x_\epsilon^{1-\theta}=\sum_{j=0}^{n-1}q_jr^j+\sum_{j=0}^{n-1}y_jr^jb,\] then \[q_j=y_j=\frac{2}{n}\cos\left(\frac{2\pi j}{n}\right)\] for every $j$. In particular, these coefficients do not depend on $\theta$.

Let \[C=\frac{2}{n},\qquad
c_j=\cos\left(\frac{2\pi j}{n}\right),\qquad d_j=\cos\left(\frac{2\pi Kj}{n}\right),\qquad
s_j=\sin\left(\frac{2\pi Kj}{n}\right). \] We have $u_j=C(c_j+\epsilon d_j)$, $v_j=C(c_j+\epsilon s_j)$, and $q_j=y_j=Cc_j$. It follows that 
\[\mathcal E_p(x_\epsilon;b)=-C^p\epsilon\sum_{j=0}^{n-1}\left(J_p(c_j+\epsilon d_j)d_j+J_p(c_j+\epsilon s_j)s_j\right).\] Because $n$ is odd, we have $c_j\neq 0$ for every $j$. Therefore,
\[\lim_{\epsilon\to 0^+} \frac{\mathcal E_p(x_\epsilon;b)}{\epsilon}=-C^p\sum_{j=0}^{n-1}J_p(c_j)(d_j+s_j). \] The terms indexed by $j$ and $n-j$ cancel in the sine sum, so
\[\sum_{j=0}^{n-1}J_p(c_j)s_j=0. \] 
By \cref{lem:Fourier}, we have \[\sum_{j=0}^{n-1}J_p(c_j)d_j<0, \] so \[\lim_{\epsilon\to 0^+}
\frac{\mathcal E_p(x_\epsilon;b)}{\epsilon}>0.\]
It follows that we can choose $\epsilon>0$ sufficiently small that $\mathcal E_p(x_\epsilon;b)>0$. 

For $\delta>0$, let $x_{\epsilon,\delta}=x_\epsilon+\delta(e-\pi)$.
This element is positive and satisfies $x_{\epsilon,\delta}\pi=0$. Moreover, it is positive definite on the kernel of $\pi$. Using the Dominated Convergence theorem, we find that 
\[\lim_{\delta\to 0^+}\mathcal E_p(x_{\epsilon,\delta};b)=\mathcal E_p(x_\epsilon;b)>0.\]
Hence, we can choose $\delta>0$ such that $\mathcal E_p(x_{\epsilon,\delta};b)>0$.

By \cref{lem:heat_realization}, there exist a symmetric generating rate collection $\blam$ on $D_n$ and a real number $R>0$ such that
\[a_1^{\blam}=e^{-R}x_{\epsilon,\delta}.\]
Because $J_p$ is homogeneous, we have \[\mathcal E_p(a_1^{\blam};b)
=e^{-Rp}\mathcal E_p(x_{\epsilon,\delta};b)>0.\]
For $\alpha\geq 0$, let $\blam^{(\alpha)}$ be obtained from $\blam$ by increasing the rate of $b$ by $\alpha$. Applying \eqref{eq:infinitesimal_converse} shows that
\[
\left.\frac{\mathrm d}{\mathrm d\alpha}\right|_{\alpha=0}\left(d_p^{\blam^{(\alpha)}}(1)\right)^p
=pe^{-Rp}\mathcal E_p(x_{\epsilon,\delta};b)>0. \]
Consequently, for every sufficiently small $\alpha>0$, we have \[
d_p^{\blam^{(\alpha)}}(1)>d_p^{\blam}(1).\] 
Since $\blam\leq\blam^{(\alpha)}$ and both rate collections are symmetric and generating, the pair $(D_n,p)$ is not rate-monotonic. 
\end{proof}

\section{Appendix: AxiomProver produced a formal Certificate}
The proofs in this paper were generated through human-AI collaboration. In dialogue with AI, the human authors developed and formalized \cref{thm:main_forward,thm:main_general,thm:main_converse}\footnote{\cref{thm:main_forward} is a special case of \cref{thm:main_general}.} with AxiomProver, an AI system currently under development by Axiom Math. In particular, this resulted in a formal Lean certificate for these three theorems. This formalization assumes standard facts from analysis and group theory. For explicit details, see 
\begin{center}
\url{https://github.com/AxiomMath/LyonsWhite}
\end{center}
This directory includes a formal challenge file containing these theorems, which can be mechanically verified using the Comparator tool in Lean.

\section*{Acknowledgments}
The authors thank Russell Lyons for introducing them to this topic.

\end{document}